\documentclass[12pt,twoside]{amsart}
\usepackage{amssymb,amscd,amsxtra,calc,mathrsfs}
\usepackage{amsmath}
\usepackage{xypic}

\title[K-stability and plt blowups]{Uniform K-stability and plt blowups of log Fano pairs}
\author{Kento Fujita} 
\date{\today}
\address{Research Institute for Mathematical Sciences, Kyoto University, Kyoto 606-8502, Japan}
\email{fujita@kurims.kyoto-u.ac.jp}

\newcommand{\pr}{\mathbb{P}}

\newcommand{\Z}{\mathbb{Z}}
\newcommand{\Q}{\mathbb{Q}}
\newcommand{\R}{\mathbb{R}}

\newcommand{\B}{{\bf B}}

\newcommand{\Supp}{\operatorname{Supp}}

\newcommand{\ord}{\operatorname{ord}}

\newcommand{\vol}{\operatorname{vol}}

\newcommand{\sO}{\mathcal{O}}

\newtheorem{thm}{Theorem}[section]
\newtheorem{lemma}[thm]{Lemma}
\newtheorem{proposition}[thm]{Proposition}
\newtheorem{corollary}[thm]{Corollary}
\newtheorem{claim}[thm]{Claim}

\theoremstyle{definition}
\newtheorem{definition}[thm]{Definition}
\newtheorem{remark}[thm]{Remark}

\newtheorem{example}[thm]{Example}
\newtheorem*{ack}{Acknowledgments}

\begin{document}

\maketitle 

\begin{abstract}
We show relationships between uniform K-stability and plt blowups of log Fano pairs. 
We see that it is enough to evaluate certain invariants defined by volume functions 
for all plt blowups in order to test uniform K-stability of log Fano pairs. 
We also discuss the uniform K-stability of two log Fano pairs under crepant finite covers. 
Moreover, we give another proof of K-semistability of the projective plane. 
\end{abstract}

\setcounter{tocdepth}{1}
\tableofcontents

\section{Introduction}\label{intro_section}

In this paper, we work over an arbitrary algebraically closed field $\Bbbk$ of 
characteristic zero. Let $(X, \Delta)$ be a \emph{log Fano pair}, that is, 
$X$ is a normal projective variety over $\Bbbk$ and $\Delta$ is an effective 
$\Q$-divisor such that $(X,\Delta)$ is a klt pair and $-(K_X+\Delta)$ is an ample 
$\Q$-Cartier $\Q$-divisor. 
(For the minimal model program, we refer the readers to \cite{KoMo} and \cite{BCHM}.)
We are interested in the problem whether 
$(X, \Delta)$ is \emph{uniformly K-stable} (resp., \emph{K-semistable}) or not 
(see \cite{tian, don, sz1, sz2, dervan, BHJ, BBJ, fjta} and references therein). 
In \cite[Theorem 3.7]{li} and \cite[Theorem 6.5]{fjta}, 
we have seen that the uniform K-stability (and the K-semistability) of 
$(X, \Delta)$ is equivalent to measure the positivity of 
certain invariants associated to divisorial valuations on $X$. 
Before recalling those results, we prepare some definitions. 

\begin{definition}\label{div_dfn}
Let $(V, \Gamma)$ be a \emph{log pair}, that is, $V$ is a normal variety and $\Gamma$ 
is an effective $\R$-divisor on $V$ such that $K_V+\Gamma$ is $\R$-Cartier. Let 
$F$ be a prime divisor \emph{over} $V$, that is, there exists a projective birational 
morphism $\sigma\colon W\to V$ with $W$ normal such that $F$ is a prime divisor 
\emph{on} $W$. 
\begin{enumerate}
\renewcommand{\theenumi}{\arabic{enumi}}
\renewcommand{\labelenumi}{(\theenumi)}
\item\label{div_dfn1}
The \emph{log discrepancy} $A(F):=A_{(V,\Gamma)}(F)$ of $(V, \Gamma)$ along $F$ 
is defined to be $A(F):=1+\ord_F(K_W-\sigma^*(K_V+\Gamma))$. 
We remark that the value $A(F)$ does not depend on the choice of $\sigma$. 
Moreover, let $c_V(F)\subset V$ be the \emph{center} of $F$ on $V$, that is, the image 
of $F$ on $V$. 
\item\label{div_dfn2}
(\cite{ishii}) 
The divisor $F$ is said to be \emph{primitive} over $V$ if there exists a projective 
birational morphism $\tau\colon T\to V$ with $T$ normal such that $F$ is a prime 
divisor on $T$ and $-F$ on $T$ is a $\tau$-ample $\Q$-Cartier divisor. 
We call the morphism $\tau$ the \emph{associated prime blowup}. 
(We do not assume that $F$ is exceptional over $V$. We remark 
that the associated prime blowup is uniquely determined from $F$ 
\cite[Proposition 2.4]{ishii}.) Moreover, let $\Gamma_T$ be the 
$\R$-divisor on $T$ defined by 
\[
K_T+\Gamma_T+(1-A(F))F=\tau^*(K_V+\Gamma).
\]
We often denote the associated prime blowup by 
\[
\tau\colon(T, \Gamma_T+F)\to(V, \Gamma).
\]
\item\label{div_dfn3}
(\cite[3.1]{shokurov}, \cite[Definition 2.1]{prokhorov})
Assume that $F$ is a primitive prime divisor over $V$ and 
$\tau\colon(T, \Gamma_T+F)\to(V, \Gamma)$ is the associated prime blowup. 
The divisor $F$ is said to be \emph{plt-type} (resp., \emph{lc-type}) 
over $(V, \Gamma)$ if the pair $(T, \Gamma_T+F)$ is plt (resp., lc). 
Under the situation, we call the associated morphism $\tau$ the \emph{associated 
plt blowup} (resp., the \emph{associated lc blowup}).
\end{enumerate}
\end{definition}

\begin{definition}[{\cite[Definition 6.1]{fjta}}]\label{F_dfn}
Let $(X, \Delta)$ be a log Fano pair of dimension $n$, $L:=-(K_X+\Delta)$, and 
let $F$ be a prime divisor over $X$.
\begin{enumerate}
\renewcommand{\theenumi}{\arabic{enumi}}
\renewcommand{\labelenumi}{(\theenumi)}
\item\label{F_dfn1}
For arbitrary $k\in\Z_{\geq 0}$ with $kL$ Cartier and $x\in\R_{\geq 0}$, let 
$H^0(X, kL-xF)$ be the sub $\Bbbk$-vector space of $H^0(X, kL)$ defined by the 
set of global sections vanishing at the generic point of $F$ at least $x$ times. 
\item\label{F_dfn2}
The divisor $F$ is said to be \emph{dreamy} over $(X, \Delta)$ 
if the graded $\Bbbk$-algebra
\[
\bigoplus_{k,i\in\Z_{\geq 0}}H^0(X, krL-iF)
\]
is finitely generated over $\Bbbk$ for some (hence, for an arbitrary) $r\in\Z_{>0}$ 
such that $rL$ is Cartier. 
\item\label{F_dfn3}
For an arbitrary $x\in \R_{\geq 0}$, set 
\[
\vol_X(L-xF):=\limsup_{\substack{k\to\infty \\ kL\text{: Cartier}}}
\frac{\dim_\Bbbk H^0(X, kL-kxF)}{k^n/n!}.
\]
By \cite{L1, L2}, the function $\vol_X(L-xF)$ is continuous and non-increasing 
on $x\in[0, \infty)$. Moreover, the limsup is actually a limit. 
\item\label{F_dfn4}
We define the \emph{pseudo-effective threshold} $\tau(F)$ of $L$ along $F$ by
\[
\tau(F):=\tau_{(X, \Delta)}(F):=\sup\{x\in\R_{>0}\,|\,\vol_X(L-xF)>0\}. 
\]
Note that $\tau(F)\in\R_{>0}$.
\item\label{F_dfn5}
We set
\[
\beta(F):=\beta_{(X, \Delta)}(F)
:=A(F)\cdot (L^{\cdot n})-\int_0^\infty\vol_X(L-xF)dx.
\]
Moreover, we set 
\[
\hat{\beta}(F):=\frac{\beta(F)}{A(F)\cdot(L^{\cdot n})}.
\] 
\item\label{F_dfn6}
We set
\[
j(F):=j_{(X, \Delta)}(F):=\int_0^{\tau(F)}\left((L^{\cdot n})-\vol_X(L-xF)\right)dx.
\]
Obviously, we have the inequality $j(F)>0$. 
\end{enumerate}
\end{definition}

\begin{remark}\label{dreamy_rmk}
\begin{enumerate}
\renewcommand{\theenumi}{\arabic{enumi}}
\renewcommand{\labelenumi}{(\theenumi)}
\item\label{dreamy_rmk1}
(\cite[Lemma 3.1 (2)]{fjtb})
If $F$ is a dreamy prime divisor over $(X, \Delta)$, then $F$ is primitive over $X$. 
\item\label{dreamy_rmk2}
(\cite{xu})
If $F$ is a plt-type prime divisor over a klt pair $(V, \Gamma)$ and $c_V(F)=\{o\}$, then 
the divisor $F$ is said to be a \emph{Koll\'ar component} of the singularity 
$o\in(V, \Gamma)$. 
\end{enumerate}
\end{remark}

The following is the valuative criterion for uniform K-stability (and K-semistability) 
of log Fano pairs 
introduced in \cite{li} and \cite{fjta}. 

\begin{thm}\label{vst_thm}
Let $(X, \Delta)$ be a log Fano pair. 
\begin{enumerate}
\renewcommand{\theenumi}{\arabic{enumi}}
\renewcommand{\labelenumi}{(\theenumi)}
\item\label{vst_thm1}
$($\cite[Theorem 3.7]{li}, \cite[Theorem 6.6]{fjta}$)$
The following are equivalent: 
\begin{enumerate}
\renewcommand{\theenumii}{\roman{enumii}}
\renewcommand{\labelenumii}{(\theenumii)}
\item\label{vst_thm11}
$(X, \Delta)$ is K-semistable $($see for example \cite[Definition 6.4]{fjta}$)$.
\item\label{vst_thm12}
For any prime divisor $F$ over $X$, the inequality $\beta(F)\geq 0$ holds.
\item\label{vst_thm13}
For any dreamy prime divisor $F$ over $(X, \Delta)$, 
the inequality $\beta(F)\geq 0$ holds.
\end{enumerate}
\item\label{vst_thm2}
$($\cite[Theorem 6.6]{fjta}$)$
The following are equivalent: 
\begin{enumerate}
\renewcommand{\theenumii}{\roman{enumii}}
\renewcommand{\labelenumii}{(\theenumii)}
\item\label{vst_thm21}
$(X, \Delta)$ is uniformly K-stable $($see for example \cite[Definition 6.4]{fjta}$)$.
\item\label{vst_thm22}
There exists $\delta\in(0,1)$ such that for any prime divisor $F$ over $X$, 
the inequality $\beta(F)\geq \delta\cdot j(F)$ holds. 
\item\label{vst_thm23}
There exists $\delta\in(0,1)$ such that for any dreamy prime divisor $F$ 
over $(X, \Delta)$, 
the inequality $\beta(F)\geq \delta\cdot j(F)$ holds. 
\end{enumerate}
\end{enumerate}
\end{thm}

The purpose of this paper is to simplify the above theorem. More precisely, 
we see relationships between the above criterion and plt blowups of log Fano pairs. 
The following is the main theorem in this paper.

\begin{thm}[Main Theorem]\label{mainthm}
Let $(X, \Delta)$ be a log Fano pair. 
\begin{enumerate}
\renewcommand{\theenumi}{\arabic{enumi}}
\renewcommand{\labelenumi}{(\theenumi)}
\item\label{mainthm1}
The following are equivalent: 
\begin{enumerate}
\renewcommand{\theenumii}{\roman{enumii}}
\renewcommand{\labelenumii}{(\theenumii)}
\item\label{mainthm11}
$(X, \Delta)$ is K-semistable.
\item\label{mainthm12}
For any plt-type prime divisor $F$ over $(X, \Delta)$, 
the inequality $\hat{\beta}(F)\geq 0$ holds.
\end{enumerate}
\item\label{mainthm2}
The following are equivalent: 
\begin{enumerate}
\renewcommand{\theenumii}{\roman{enumii}}
\renewcommand{\labelenumii}{(\theenumii)}
\item\label{mainthm21}
$(X, \Delta)$ is uniformly K-stable.
\item\label{mainthm22}
There exists $\varepsilon\in(0,1)$ such that for any prime divisor $F$ over $X$, 
the inequality $\hat{\beta}(F)\geq \varepsilon$ holds. 
\item\label{mainthm23}
There exists $\varepsilon\in(0,1)$ such that for any dreamy prime divisor $F$ 
over $(X, \Delta)$, the inequality $\hat{\beta}(F)\geq \varepsilon$ holds. 
\item\label{mainthm24}
There exists $\varepsilon\in(0,1)$ such that for any plt-type prime divisor $F$ 
over $(X, \Delta)$, 
the inequality $\hat{\beta}(F)\geq \varepsilon$ holds. 
\end{enumerate}
\end{enumerate}
\end{thm}

\begin{remark}\label{main_rmk}
\begin{enumerate}
\renewcommand{\theenumi}{\arabic{enumi}}
\renewcommand{\labelenumi}{(\theenumi)}
\item\label{main_rmk1}
In Theorem \ref{vst_thm} \eqref{vst_thm2}, we need to evaluate $j(F)$ in order to 
test uniform K-stability. It seems relatively difficult to evaluate $j(F)$ than 
to evaluate $\hat{\beta}(F)$ since 
the value $\tau(F)$ is not easy to treat. 
It is one of the remarkable point that 
we do not need to evaluate $j(F)$ in Theorem \ref{mainthm} \eqref{mainthm2}. 
\item\label{main_rmk2}
Theorem \ref{mainthm} claims that we can check uniform K-stability and K-semistability 
by evaluating $\hat{\beta}(F)$ for plt-type prime divisors $F$ over $(X, \Delta)$. 
The theory of plt blowups is important for the theory of minimal model program 
and singularity theory (see \cite{prokhorov, prokhorov_MSJ}). 
It is interesting that such theories 
will relate K-stability via Theorem \ref{mainthm}. 
Moreover, Theorem \ref{mainthm} seems to relate with \cite[Conjecture 6.5]{LX}.
\end{enumerate}
\end{remark}

As an easy consequence of Theorem \ref{mainthm}, we get the following result. 
The proof is given in Section \ref{finite_section}. We remark that Dervan also 
treated similar problem. See \cite{dervan_finite}. 

\begin{corollary}[{see also Example \ref{finite_ex}}]\label{finite_cor}
Let $(X, \Delta)$ and $(X', \Delta')$ be log Fano pairs. Assume that there exists a 
finite and surjective morphism $\phi\colon X'\to X$ such that 
$\phi^*(K_X+\Delta)=K_{X'}+\Delta'$. If $(X', \Delta')$ is uniformly K-stable $($resp., 
K-semistable$)$, then so is $(X, \Delta)$. 
\end{corollary}

We can also show as an application of Theorem \ref{mainthm} that 
the projective plane is K-semistable. The result is well-known (see \cite{don}). 
Moreover, the result has been already proved purely 
algebraically (see \cite{kempf, RT} and \cite{li, blum, PW}). 
However, it is worth writing the proof since 
our proof is purely birational geometric. The proof is given in Section \ref{P2_section}.

\begin{corollary}[{see also \cite{kempf,don,li,blum,PW}}]\label{P2_cor}
The projective plane $\pr^2$ is $($that is, the pair $(\pr^2, 0)$ is$)$ 
K-semistable.
\end{corollary}

This paper is organized as follows. In Section \ref{tau_section}, we see the equivalence 
between the conditions in Theorem \ref{mainthm} \eqref{mainthm21}, 
Theorem \ref{mainthm} \eqref{mainthm22}, and Theorem \ref{mainthm} \eqref{mainthm23}. For the proof, we use the log-convexity of volume functions 
and restricted volume functions. 
In Section \ref{plt_section}, we see how to replace a primitive divisor by a plt-type 
prime divisor with smaller $\hat{\beta}$-invariant. For the proof, we use techniques of 
minimal model program. Theorem \ref{mainthm} follows from those observations. 
In Section \ref{app_section}, we prove Corollaries \ref{finite_cor} and \ref{P2_cor}.

\begin{ack}
The author thank Doctor Atsushi Ito and Professor Shunsuke Takagi 
for discussions during the author enjoyed the summer 
school named ``Algebraic Geometry Summer School 2016" in Tambara Institute of 
Mathematical Sciences. 
This work was supported by JSPS KAKENHI Grant Number JP16H06885.
\end{ack}

\section{Uniform K-stability}\label{tau_section}

In this section, we simplify the conditions in Theorem \ref{vst_thm} \eqref{vst_thm2}. 
In this section, we always assume that 
$(X, \Delta)$ is a log Fano pair of dimension $n$, $L:=-(K_X+\Delta)$, and 
$F$ is a prime divisor over $X$. 

The proof of the following proposition is essentially same as the proofs of 
\cite[Theorem 4.2]{FO} and \cite[Proposition 3.2]{fjtb}.

\begin{proposition}\label{tau_prop}
We have the inequality 
\[
\frac{n}{n+1}\tau(F)\geq\frac{1}{(L^{\cdot n})}\int_0^\infty\vol_X(L-xF)dx.
\]
\end{proposition}

\begin{proof}
Take any log resolution $\sigma\colon Y\to X$ of $(X, \Delta)$ such that 
$F\subset Y$ and there exists a $\sigma$-ample $\Q$-divisor $A_Y$ on $Y$ with 
$\gamma:=-\ord_F A_Y>0$ and $-A_Y$ effective. 
Then, for any $0<\varepsilon\ll 1$, 
$\sigma^*L+(\varepsilon/\gamma) A_Y$ is ample. Hence 
$\B_+(\sigma^*L+(\varepsilon/\gamma) A_Y)=\emptyset$, 
where $\B_+$ is the augmented base locus (see \cite{ELMNP}). Note that  
\[
\B_+(\sigma^*L-\varepsilon F)\subset\B_+(\sigma^*L+(\varepsilon/\gamma) A_Y)
\cup\Supp(-(\varepsilon/\gamma)A_Y-\varepsilon F).
\]
This implies that $F\not\subset\B_+(\sigma^*L-\varepsilon F)$. Thus, 
by \cite[Theorem A]{ELMNP} (and by \cite[Theorem A and Corollary C]{BFJ}), 
the restricted volume $\vol_{Y|F}(\sigma^*L-xF)$ 
on $x\in[0, \tau(F))$ satisfies the log-concavity (in the sense of 
\cite[Theorem A]{ELMNP}). 
In particular, for an arbitrary $x_0\in(0, \tau(F))$, we have 
\[
\begin{cases}
\vol_{Y|F}(\sigma^*L-xF)\geq(x/x_0)^{n-1}\cdot\vol_{Y|F}(\sigma^*-x_0F) & \text{if }
x\in[0,x_0],\\
\vol_{Y|F}(\sigma^*L-xF)\leq(x/x_0)^{n-1}\cdot\vol_{Y|F}(\sigma^*-x_0F) & \text{if }
x\in[x_0,\tau(F)).
\end{cases}
\]
On the other hand, by \cite[Corollary 4.27]{LM}, for an arbitrary $x\in[0, \tau(F)]$, 
we have the equality 
\[
\vol_Y(\sigma^*L-xF)=n\int_x^{\tau(F)}\vol_{Y|F}(\sigma^*L-yF)dy.
\]

Let us set 
\[
b:=\frac{\int_0^{\tau(F)}y\cdot\vol_{Y|F}(\sigma^*L-yF)dy}{\int_0^{\tau(F)}
\vol_{Y|F}(\sigma^*L-yF)dy}.
\]
Obviously, $b\in(0, \tau(F))$ holds. Moreover, we get
\begin{eqnarray*}
0&=&\int_{-b}^{\tau(F)-b}y\cdot\vol_{Y|F}(\sigma^*L-(y+b)F)dy\\
&\leq&\int_{-b}^{\tau(F)-b}y\cdot\left(\frac{y+b}{b}\right)^{n-1}\cdot
\vol_{Y|F}(\sigma^*L-bF)dy\\
&=&\frac{\vol_{Y|F}(\sigma^*L-bF)\cdot\tau(F)^n}
{n\cdot b^{n-1}}\left(\frac{n}{n+1}\tau(F)-b\right).
\end{eqnarray*}
Thus $b\leq (n/(n+1))\tau(F)$. On the other hand, we have 
\begin{eqnarray*}
b&=&\frac{n\int_0^{\tau(F)}\int_x^{\tau(F)}\vol_{Y|F}(\sigma^*L-yF)dydx}{(L^{\cdot n})}\\
&=&\frac{\int_0^{\tau(F)}\vol_Y(\sigma^*L-xF)dx}{(L^{\cdot n})}.
\end{eqnarray*}
Thus we have proved Proposition \ref{tau_prop}. 
\end{proof}

The following lemma is nothing but a logarithmic version of \cite[Lemma 2.2]{FO}. 
We give a proof just for the readers' convenience. 

\begin{lemma}[{\cite[Lemma 2.2]{FO}}]\label{FO_lem}
We have the inequality
\[
\frac{\tau(F)}{n+1}\leq\frac{1}{(L^{\cdot n})}\int_0^\infty\vol_X(L-xF)dx.
\]
\end{lemma}

\begin{proof}
By \cite[Corollary 4.12]{LM}, we have 
\[
\vol_X(L-xF)\geq\left(1-\frac{x}{\tau(F)}\right)^n\cdot(L^{\cdot n}).
\]
Lemma \ref{FO_lem} follows immediately from the above.
\end{proof}

Now we are ready to prove the following theorem. 

\begin{thm}\label{j_thm}
Let $(X, \Delta)$ be a log Fano pair. Then the following are equivalent: 
\begin{enumerate}
\renewcommand{\theenumi}{\roman{enumi}}
\renewcommand{\labelenumi}{(\theenumi)}
\item\label{j_thm1}
$(X, \Delta)$ is uniformly K-stable.
\item\label{j_thm2}
There exists $\varepsilon\in(0,1)$ such that for any prime divisor $F$ over $X$, 
the inequality $\hat{\beta}(F)\geq \varepsilon$ holds. 
\item\label{j_thm3}
There exists $\varepsilon\in(0,1)$ such that for any dreamy prime divisor $F$ 
over $(X, \Delta)$, the inequality $\hat{\beta}(F)\geq \varepsilon$ holds. 
\end{enumerate}
\end{thm}

\begin{proof}
Let $F$ be an arbitrary prime divisor over $X$. 
Firstly, we observe that the condition $\beta(F)\geq \delta\cdot j(F)$ 
for some $\delta\in(0, 1)$ is equivalent to the condition 
\begin{equation}\label{one}
(1+\delta')A(F)-\delta'\tau(F)\geq\frac{1}{(L^{\cdot n})}\int_0^\infty\vol_X(L-xF)dx,
\end{equation}
where $\delta':=\delta/(1-\delta)\in(0,\infty)$.

We also observe that the condition $\hat{\beta}(F)\geq \varepsilon$ for some 
$\varepsilon\in(0,1)$ is equivalent to the condition 
\begin{equation}\label{two}
A(F)\geq\frac{1+\varepsilon'}{(L^{\cdot n})}\int_0^\infty\vol_X(L-xF)dx,
\end{equation}
where $\varepsilon':=\varepsilon/(1-\varepsilon)\in(0,\infty)$. 

\begin{claim}[{see \cite[Theorem 2.3]{FO}}]\label{twoone_claim}
If the inequality \eqref{two} holds for some $\varepsilon'\in(0,\infty)$, then 
the inequality \eqref{one} holds for $\delta'=\varepsilon'/(n+1)$.
\end{claim}

\begin{proof}[Proof of Claim \ref{twoone_claim}]
By Lemma \ref{FO_lem}, the inequality \eqref{two} implies that 
\[
A(F)\geq \frac{1}{(L^{\cdot n})}\int_0^\infty\vol_X(\sigma^*L-xF)dx+
\frac{\varepsilon'}{n+1}\tau(F).
\]
Thus the inequality \eqref{one} holds for $\delta'=\varepsilon'/(n+1)$ since $A(F)>0$. 
\end{proof}

\begin{claim}\label{onetwo_claim}
If the inequality \eqref{one} holds for some $\delta'\in(0,\infty)$, then the inequality 
\eqref{two} holds for 
\[
\varepsilon'=\min\left\{\frac{\delta'\frac{1-\theta}{\theta}}{1-
\delta'\frac{1-\theta}{\theta}},\,\, \frac{1}{2n+1}\right\}\in(0,1), 
\]
where 
\begin{eqnarray*}
\theta&:=&\max\left\{\frac{2n}{2n+1},\,\, \frac{2\delta'}{2\delta'+1}\right\}\in(0,1).
\end{eqnarray*}
\end{claim}

\begin{proof}[Proof of Claim \ref{onetwo_claim}]
We firstly assume that case $A(F)<\theta\cdot\tau(F)$. Then the inequality \eqref{one} 
implies that 
\[
\frac{1}{(L^{\cdot n})}\int_0^\infty\vol_X(L-xF)dx
<\left(1-\delta'\frac{1-\theta}{\theta}\right)A(F).
\]
Note that 
\[
\delta'\frac{1-\theta}{\theta}\in(0, 1/2]\subset(0,1).
\]
Thus the inequality \eqref{two} holds for such $\varepsilon'$. 

We secondly consider the remaining case  $A(F)\geq \theta\cdot\tau(F)$. 
In this case, by Proposition \ref{tau_prop}, we have 
\[
\frac{1}{(L^{\cdot n})}\int_0^\infty\vol_X(L-xF)dx\leq\frac{n}{n+1}\tau(F)
\leq\frac{n}{n+1}\frac{1}{\theta}A(F).
\]
Note that 
\[
\frac{n+1}{n}\theta-1\geq \frac{1}{2n+1}.
\]
Thus the inequality \eqref{two} holds for such $\varepsilon'$. 
\end{proof}

Theorem \ref{j_thm} immediately follows from Theorem \ref{vst_thm} \eqref{vst_thm2}, 
Claims \ref{twoone_claim} and \ref{onetwo_claim}.
\end{proof}

\begin{remark}\label{big_rmk}
Let $(X, \Delta)$ be a log Fano pair and $F$ be a prime divisor over $X$. 
If $\tau(F)\leq A(F)$, then $\hat{\beta}(F)\geq 1/(n+1)$ by Proposition \ref{tau_prop}. 
Thus it is enough to consider prime divisors $F$ over $X$ with $\tau(F)>A(F)$ 
in order to check the conditions in Theorem \ref{mainthm}. 
\end{remark}

\section{Plt blowups}\label{plt_section}

The following theorem is inspired by \cite[Lemma 1]{xu}. 

\begin{thm}\label{mmp_thm}
Let $(X, \Delta)$ be a quasi-projective klt pair with $\Delta$ effective $\R$-divisor. 
Let $F$ be a primitive prime divisor over $X$ and 
$\sigma\colon(Y, \Delta_Y+F)\to(X, \Delta)$ be the associated prime blowup. 
Assume that $F$ is not plt-type $($resp., not lc-type$)$ over $(X, \Delta)$. 
Then there exists a plt-type prime divisor $G$ over $(X, \Delta)$ such that 
the inequality $A_{(Y, \Delta_Y+F)}(G)\leq 0$ $($resp., $A_{(Y, \Delta_Y+F)}(G)< 0)$ holds.
\end{thm}

\begin{proof}
Let $\pi\colon V\to Y$ be a log resolution of $(Y, \Delta_Y+F)$ and let 
$E_1,\dots,E_k$ be the set of $\pi$-exceptional divisors on $V$. 
We set $F_V:=\pi^{-1}_*F$ and $\Delta_V:=\pi^{-1}_*\Delta_Y$. We may assume that 
there exists a $(\sigma\circ\pi)$-ample $\Q$-divisor 
\[
A_V=-\sum_{i=1}^kh_iE_i-h_FF_V
\]
on $V$ with $h_1,\dots,h_k\in\Q_{>0}$ and 
\[
\begin{cases}
h_F\in\Q_{>0} & \text{if $F$ is exceptional over $X$}, \\
h_F=0 & \text{otherwise}.
\end{cases}
\]
Take a sufficiently ample Cartier divisor $L$ on $X$ such that $\sigma^*L-F$ is ample. 
Let $L_Y$ be a general effective $\Q$-divisor with small coefficients such that 
$L_Y\sim_\Q\sigma^*L-F$. Set $L_V:=\pi^{-1}_*L_Y$. Then 
$L_V=\pi^*L_Y$ and the pair $(V, \Delta_V+F_V+\sum_{i=1}^kE_i+L_V)$ is log smooth 
by generality. Moreover, the $\Q$-divisor $\sigma_*(L_Y+F)$ is $\Q$-Cartier with 
$\sigma_*(L_Y+F)\sim_\Q L$ and $\sigma^*\sigma_*(L_Y+F)=L_Y+F$. 

Let us set 
\begin{eqnarray*}
\pi^*F &=: & F_V+\sum_{i=1}^kc_iE_i\quad(c_i\in\Q_{\geq 0}), \\
a_F & := & A_{(X, \Delta)}(F)\in\R_{>0},\\
a_i & := & A_{(X, \Delta)}(E_i)\in\R_{>0},\\
b_i & := & A_{(Y, \Delta_Y+F)}(E_i)\in \R.
\end{eqnarray*}
Of course, we have the inequality $h_i>h_Fc_i$ (from the negativity lemma) and 
the equality $a_i=b_i+a_Fc_i$ for any $1\leq i\leq k$. 
By assumption, the inequality $b_i\leq 0$ (resp., $b_i<0$) holds for some $1\leq i\leq k$.
In particular, the inequality $c_i>0$ holds for some $1\leq i\leq k$. 
By changing $E_1,\dots,E_k$ and by perturbing the coefficients of $A_V$ if necessary, 
we can assume that the following conditions are satisfied: 
\begin{itemize}
\item
There exists $1\leq l\leq k$ such that $c_i>0$ holds if and only if $1\leq i\leq l$.
\item
$b_1/c_1=\min_{1\leq i\leq l}\{b_i/c_i\}.$
\item
The inequality $h_j/c_j<h_1/c_1$ holds for any $2\leq j\leq l$ with $b_1/c_1=b_j/c_j$.
\end{itemize}
Take a rational number $0<\varepsilon\ll 1$ and set 
$t:=(a_1-\varepsilon h_1)/c_1\in\R_{>0}$. Since $\varepsilon$ is very small, we get 
the following properties: 
\begin{itemize}
\item
$a_1-(tc_1+\varepsilon h_1)=0$, 
\item
$a_i-(tc_i+\varepsilon h_i)>0$ holds for any $2\leq i\leq k$, 
\item
$1-a_F+t\in(1-a_F, 1)$ and $1-a_F+t+\varepsilon h_F<1$, and 
\item
$\pi^*(L_Y+F)+\varepsilon A_V$ is ample on $V$.
\end{itemize}

Take a general effective $\R$-divisor $L'$ with small coefficients such that 
$L'\sim_\R\pi^*(L_Y+F)+\varepsilon A_V$. Moreover, we set
\[
b_F:=\begin{cases}
1 & \text{if $F$ is exceptional over $X$}, \\
1-a_F+t & \text{otherwise}.
\end{cases}
\]
Then 
\[
K_V+\Delta_V+b_FF_V+tL_V+\sum_{i=1}^kE_i+L'
\] 
is dlt. 
(We remark that $1-a_F+t\geq t>0$ holds if $F$ is a divisor on $X$. 
We also remark that the coefficients of $tL_V$ can be less than $1$ by the definition 
of $L_V$.)
Moreover, we have 
\begin{eqnarray*}
&&K_V+\Delta_V+b_FF_V+tL_V+\sum_{i=1}^kE_i+L'\\
&\sim_\R&(\sigma\circ\pi)^*(K_X+\Delta)+t\pi^*(L_Y+F)+(a_F+b_F-1)F_V-t\pi^*F\\
&&+\sum_{i=1}^ka_iE_i+\pi^*(L_Y+F)+\varepsilon A_V\\
&\sim_{\R,X}&\left(b_F-(1-a_F+t)-\varepsilon h_F\right)F_V+\sum_{i=2}^k
\left(a_i-(tc_i+\varepsilon h_i)\right)E_i.
\end{eqnarray*}
The right-hand side is effective and its support is equal to the union of 
$(\sigma\circ\pi)$-exceptional prime divisors other than $E_1$.
Furthermore,
\begin{eqnarray*}
&&K_V+\Delta_V+b_FF_V+tL_V+\sum_{i=1}^kE_i+L'\\
&\sim_{\R,X}&K_V+\Delta_V+b_FF_V+tL_V+\sum_{i=1}^kE_i+(1-\delta)L'+\delta\varepsilon
A_V
\end{eqnarray*}
is klt for $0<\delta\ll 1$. Thus, by \cite[Corollary 1.4.2]{BCHM}, 
we can run and terminate 
a $(K_V+\Delta_V+b_FF_V+tL_V+\sum_{i=1}^kE_i+L')$-MMP with scaling $L'$ over $X$. 
Let 
\[\xymatrix{
V \ar@{-->}[rr]^{\psi} \ar[dr]_{\sigma\circ\pi} &  & W \ar[dl]^{\phi}\\
& X & \\
}\]
be the output of this MMP. The MMP does not contract $E_1$. Let $G_W\subset W$ be 
the image of $E_1$. Moreover, by the negativity lemma, any 
$(\sigma\circ\pi)$-exceptional prime divisor other than $E_1$ is contracted by $\psi$. 
In particular, we get 
\[
\psi_*(K_V+\Delta_V+b_FF_V+tL_V+\sum_{i=1}^kE_i+L')\sim_{\R, X}0.
\]
Furthermore, by the definition of MMP with scaling, the $\R$-divisor
\begin{eqnarray*}
&&\psi_*(K_V+\Delta_V+b_FF_V+tL_V+\sum_{i=1}^kE_i+(1+\lambda)L')\\
&&\sim_{\R, X}
\lambda\psi_*L'\sim_{\R, X}-\lambda\varepsilon h_1G_W
\end{eqnarray*}
is nef over $X$ for any $0<\lambda\ll 1$. 
Moreover, by the base point free theorem, the above $\R$-divisor admits 
the ample model over $X$. Let
\[\xymatrix{
W \ar[rr]^{\mu} \ar[dr]_{\phi} &  & Z \ar[dl]^{\tau}\\
& X & \\
}\]
be the model and we set $G:=\mu_*G_W$. Since $-G$ is $\tau$-ample, the morphism 
$\mu$ is a small morphism. We remark that 
\begin{eqnarray*}
&&\psi_*(K_V+\Delta_V+b_FF_V+tL_V+\sum_{i=1}^kE_i+L')\\
&=&K_W+\phi^{-1}_*\Delta+t\psi_*(F_V+L_V)+G_W+\psi_*L'
\end{eqnarray*}
is dlt, $\R$-linearly equivalent to zero over $X$, and $G_W$ is the unique prime divisor 
whose coefficient is equal to one. Thus this is plt and 
\[
K_Z+\tau^{-1}_*\Delta+t(\mu\circ\psi)_*(F_V+L_V)+G+(\mu\circ\psi)_*L'
\]
is also plt. Note that $(\mu\circ\psi)_*L'$ is effective $\R$-Cartier. Moreover, since 
$\tau_*(\mu\circ\psi)_*(F_V+L_V)$ is $\R$-Cartier and 
$\tau^*\tau_*(\mu\circ\psi)_*(F_V+L_V)-(\mu\circ\psi)_*(F_V+L_V)$ is equal to 
some multiple of $G$, the $\R$-divisor $(\mu\circ\psi)_*(F_V+L_V)$ is also 
effective $\R$-Cartier. This implies that the pair $(Z, \tau^{-1}_*\Delta+G)$ is also plt. 
By construction, $-G$ is $\tau$-ample, $G$ is exceptional over $X$ (since 
$E_1$ is exceptional over $Y$), and $A_{(Y, \Delta_Y+F)}(G)=b_1\leq 0$ (resp., $<0$). 
\end{proof}

\begin{corollary}\label{mmp_cor}
Let $(X, \Delta)$ be a log Fano pair and $F$ be a primitive prime divisor over $X$ 
with the associated prime blowup $\sigma\colon (Y, \Delta_Y+F)\to(X, \Delta)$. 
If $F$ is not plt-type over $(X, \Delta)$, then there exists a plt-type prime divisor $G$ 
over $(X, \Delta)$ with $A_{(Y, \Delta_Y+F)}(G)\leq 0$ such that the inequality 
$\hat{\beta}(F)>\hat{\beta}(G)$ holds. 
\end{corollary}

\begin{proof}
We set $n:=\dim X$ and $L:=-(K_X+\Delta)$. By Theorem \ref{mmp_thm}, there exists 
a plt-type prime divisor $G$ over $(X, \Delta)$ with $A_{(Y, \Delta_Y+F)}(G)\leq 0$. 
Let $\tau\colon (Z, \Delta_Z+G)\to (X, \Delta)$ be the associated plt blowup. Let 
\[\xymatrix{
& V \ar[dl]_{\pi} \ar[dr]^{\rho}&\\
Y \ar[dr]_{\sigma} &  & Z \ar[dl]^{\tau}\\
& X & \\
}\]
be a common log resolution of $(Y, \Delta_Y+F)$ and $(Z, \Delta_Z+G)$. 
Since $(X, \Delta)$ is klt and $A_{(Y, \Delta_Y+F)}(G)\leq 0$, we have 
$c_Y(G)\subset F$. Set $f:=\ord_{\pi^{-1}_*F}(\rho^*G)$ and 
$g:=\ord_{\rho^{-1}_*G}(\pi^*F)$. Since $c_Y(G)\subset F$, we have the inequality 
$g>0$. Moreover, we have the following equalities: 
\begin{eqnarray*}
A_{(X, \Delta)}(F) &=& A_{(Z, \Delta_Z+G)}(F)+f\cdot A_{(X, \Delta)}(G),\\
A_{(X, \Delta)}(G) &=& A_{(Y, \Delta_Y+F)}(G)+g\cdot A_{(X, \Delta)}(F).
\end{eqnarray*}

\begin{claim}\label{mmp_claim}
\begin{enumerate}
\renewcommand{\theenumii}{\roman{enumii}}
\renewcommand{\labelenumii}{(\theenumii)}
\item\label{mmp_claim1}
For any $x\in\R_{\geq 0}$, we have the inequality 
$\vol_X(L-xF)\leq \vol_X(L-gxG)$.
\item\label{mmp_claim2}
If $A_{(Y, \Delta_Y+F)}(G)=0$, then, for any $0<x\ll 1$, we have the inequality 
$\vol_X(L-xF)< \vol_X(L-gxG)$.
\end{enumerate}
\end{claim}

\begin{proof}[Proof of Claim \ref{mmp_claim}]
The assertion \eqref{mmp_claim1} is trivial since we know that 
\[
H^0(X, kL-jF)\subset H^0(X, kL-gjG)
\]
for any sufficiently divisible $k$, $j\in\Z_{>0}$. We see the assertion \eqref{mmp_claim2}. 
We assume that $A_{(Y, \Delta_Y+F)}(G)=0$. Then we have 
\begin{eqnarray*}
&&A_{(X, \Delta)}(G)=gA_{(X, \Delta)}(F)\\
&=&gA_{(Z, \Delta_Z+G)}(F)+fgA_{(X, \Delta)}(G)
>fgA_{(X, \Delta)}(G).
\end{eqnarray*}
This implies the inequality $1>fg$. 

Fix any $0<x\ll 1$ such that both $\sigma^*L-xF$ and $\tau^*L-gxG$ are ample. 
Note that $\vol_X(L-xF)=((\sigma^*L-xF)^{\cdot n})$ and 
$\vol_X(L-gxG)=((\tau^*L-gxG)^{\cdot n})$.
Since 
\[
-x(\pi^*F-g\rho^*G)=\pi^*(\sigma^*L-xF)-\rho^*(\tau^*L-gxG)
\]
is $\rho$-nef and 
\[
\rho_*(\pi^*F-g\rho^*G)=(1-fg)\rho_*\pi^{-1}_*F\geq 0, 
\]
we have $\pi^*F\geq g\rho^*G$ by the negativity lemma. 
Thus, for any $0\leq i\leq n-1$, we have 
\begin{eqnarray*}
0&\leq& \left(\pi^*(\sigma^*L-xF)^{\cdot i}\cdot \rho^*(\tau^*L-gxG)^{\cdot n-1-i}
\cdot \pi^*(xF)-\rho^*(gxG)\right)\\
&=&\left(\pi^*(\sigma^*L-xF)^{\cdot i}\cdot \rho^*(\tau^*L-gxG)^{\cdot n-i}\right)\\
&&
-\left(\pi^*(\sigma^*L-xF)^{\cdot i+1}\cdot \rho^*(\tau^*L-gxG)^{\cdot n-1-i}\right).
\end{eqnarray*}
Moreover, we have
\begin{eqnarray*}
&&\left(\pi^*(\sigma^*L-xF)^{\cdot n-1}\cdot \rho^*(\tau^*L-gxG)\right)
-\left(\pi^*(\sigma^*L-xF)^{\cdot n}\right)\\
&=&x(1-fg)\left((\sigma^*L-xF)^{\cdot n-1}\cdot F\right)>0.
\end{eqnarray*}
Therefore, we have 
\begin{eqnarray*}
&&((\tau^*L-gxG)^{\cdot n})-((\sigma^*L-xF)^{\cdot n})\\
&=&\sum_{i=0}^{n-1}\biggl(
\left(\pi^*(\sigma^*L-xF)^{\cdot i}\cdot \rho^*(\tau^*L-gxG)^{\cdot n-i}\right)\\
&&
-\left(\pi^*(\sigma^*L-xF)^{\cdot i+1}\cdot \rho^*(\tau^*L-gxG)^{\cdot n-1-i}\right)
\biggr)>0.
\end{eqnarray*}
Thus we get Claim \ref{mmp_claim}.
\end{proof}

From Claim \ref{mmp_claim}, we get the inequalities 
\begin{eqnarray*}
\hat{\beta}(F) &=& 1-\frac{\int_0^\infty\vol_X(L-xF)dx}{A_{(X, \Delta)}(F)
\cdot(L^{\cdot n})}\\
&\geq& 1-\frac{\int_0^\infty\vol_X(L-gxG)dx}{A_{(X, \Delta)}(F)
\cdot(L^{\cdot n})}=1-\frac{\int_0^\infty\vol_X(L-xG)dx}{gA_{(X, \Delta)}(F)
\cdot(L^{\cdot n})}\\
&\geq&1-\frac{\int_0^\infty\vol_X(L-xG)dx}{A_{(X, \Delta)}(G)
\cdot(L^{\cdot n})}=\hat{\beta}(G).
\end{eqnarray*}
Moreover, at least one of the inequalities is the strict inequality.
\end{proof}

\begin{proof}[Proof of Theorem \ref{mainthm}]
This follows immediately from Remark \ref{dreamy_rmk} \eqref{dreamy_rmk1}, 
Theorems \ref{vst_thm}, 
\ref{j_thm} and Corollary \ref{mmp_cor}.
\end{proof}

\section{Applications}\label{app_section}

In this section, we give several applications of Theorem \ref{mainthm}. 

\subsection{Finite covers}\label{finite_section}

In this section, we prove Corollary \ref{finite_cor}. 
To begin with, we show the following easy lemma. 

\begin{lemma}\label{big_vol_lem}
Let $\psi\colon W\to V$ be a generically finite and surjective morphism between normal 
projective varieties. For any Cartier divisor $D$ on $V$, we have the following inequality: 
\[
\vol_W(\psi^*D)\geq(\deg \psi)\cdot\vol_V(D).
\]
\end{lemma}

\begin{proof}
We may assume that $D$ is big. 
By \cite[Theorem]{fujita}, 
for any $\varepsilon>0$, there exists a projective birational morphism 
$\sigma\colon V'\to V$ with $V'$ normal, ample $\Q$-divisor $A$, and an effective 
$\Q$-divisor $E$ such that $\sigma^*D\sim_\Q A+E$ and $\vol_V(D)\leq\vol_{V'}(A)+
\varepsilon$ hold. Let 
\[
\begin{CD}
W'       @>{\psi'}>>  V'          \\
@V{\sigma'}VV   @VV{\sigma}V    \\
W     @>>{\psi}>  V
\end{CD}
\]
be the normalization of the fiber product. Then we get 
\begin{eqnarray*}
&&\vol_W(\psi^*D)=\vol_{W'}(\psi'{}^*\sigma^*D)\geq\vol_{W'}(\psi'{}^*A)\\
&=&(\deg \psi)\cdot\vol_{V'}(A)\geq (\deg \psi)\cdot(\vol_V(D)-\varepsilon).
\end{eqnarray*}
The assertion immediately follows from the above inequalities. 
\end{proof}

\begin{proof}[Proof of Corollary \ref{finite_cor}]
We set $n:=\dim X$, $L:=-(K_X+\Delta)$, $L':=-(K_{X'}+\Delta')$ and $d:=\deg \phi$. 
From Theorem \ref{mainthm}, there exists $\varepsilon>0$ (resp., $\geq 0$) such that 
$\hat{\beta}(F')\geq \varepsilon$ holds for any prime divisor $F'$ over $X'$. 
Take any plt blowup $\sigma\colon(Y, \Delta_Y+F)\to(X, \Delta)$. 
From Theorem \ref{mainthm}, it is enough to show the inequality $\hat{\beta}(F)\geq 
\varepsilon$. Let 
\[
\begin{CD}
Y'       @>{\psi}>>  Y          \\
@V{\sigma'}VV  @VV{\sigma}V    \\
X'     @>>{\phi}>    X
\end{CD}
\]
be the normalization of the fiber product. 
(Note that the morphism $\psi$ is a finite morphism.) 
Let 
\[
\psi^*F=\sum_{i=1}^mr_iF'_i
\]
be the irreducible decomposition of the pullback of $F$ 
(see \cite[Proposition 5.20]{KoMo}), where $r_i\in\Z_{>0}$. 
By \cite[Proposition 5.20]{KoMo}, we have the equality 
\[
A_{(X', \Delta')}(F'_i)=r_i A_{(X, \Delta)}(F)
\]
for any $1\leq i\leq m$. Moreover, for any $x\in\R_{\geq 0}$, we have 
\[
\psi^*(\sigma^*L-xF)=\sigma'{}^*L'-x\sum_{i=1}^mr_iF'_i\leq\sigma'{}^*L'-xr_1F'_1. 
\]
Therefore, from Lemma \ref{big_vol_lem}, we have the following inequalities: 
\begin{eqnarray*}
1-\varepsilon&\geq&
 1-\hat{\beta}(F'_1)=\frac{\int_0^\infty\vol_{Y'}(\sigma'{}^*L'-yF'_1)dy}{A_{(X', \Delta')}
(F'_1)\cdot(L'{}^{\cdot n})}\\
&=&\frac{\int_0^\infty\vol_{Y'}(\sigma'{}^*L'-xr_1F'_1)dx}
{A_{(X, \Delta)}(F)\cdot d(L{}^{\cdot n})}\\
&\geq&\frac{\int_0^\infty\vol_{Y'}(\psi^*(\sigma^*L-xF))dx}
{A_{(X, \Delta)}(F)\cdot d(L{}^{\cdot n})}\\
&\geq&\frac{\int_0^\infty\vol_{Y}(\sigma^*L-xF)dx}
{A_{(X, \Delta)}(F)\cdot (L{}^{\cdot n})}=1-\hat{\beta}(F). 
\end{eqnarray*}
As a consequence, we have proved Corollary \ref{finite_cor}. 
\end{proof}

We remark that the converse of Corollary \ref{finite_cor} is not true in general. 
See the following example. 

\begin{example}\label{finite_ex}
Let $X:=\pr^1$, $X':=\pr^1$ and let us consider the morphism $\phi\colon X'\to X$ with 
$t\mapsto t^2$. Take any $d\in(0,1)\cap\Q$ and set 
\[
\begin{cases}
\Delta:=\frac{1}{2}[0]+\frac{1}{2}[\infty]+d[1] & \text{on }X, \\
\Delta':=d[1]+d[-1] & \text{on }X'
.\end{cases}
\]
Then we know that $(X, \Delta)$ and $(X', \Delta')$ are log Fano pairs 
and the equality $\phi^*(K_X+\Delta)=K_{X'}+\Delta'$ holds from the ramification formula. 
By \cite[Example 6.6]{fjta}, $(X, \Delta)$ is uniformly K-stable. However, again by 
\cite[Example 6.6]{fjta}, $(X', \Delta')$ is not uniformly K-stable (but K-semistable). 
\end{example}

\subsection{K-semistability of the projective plane}\label{P2_section}

In this section, we show Corollary \ref{P2_cor}. 
Take any plt blowup $\sigma\colon(Y,F)\to(\pr^2,0)$. 
It is enough to show the inequality 
$\hat{\beta}(F)\geq 0$ by Theorem \ref{mainthm}. 

Assume that $F$ is a divisor on $\pr^2$. Set $d:=\deg_{\pr^2}F$. Then 
\[
\vol_{\pr^2}(-K_{\pr^2}-xF)=(3-dx)^2
\]
for $x\in[0,3/d]$ and $A(F)=1$. Thus we get $\hat{\beta}(F)=(d-1)/d\geq 0$. 
(See also \cite[Corollary 9.3]{div_stability}.)

From now on, we assume that $F$ is an exceptional divisor over $\pr^2$. 
Set $\{p\}:=c_{\pr^2}(F)$. Of course, the Picard rank of $Y$ is equal to two. 
We may assume that the divisor $-(K_Y+F)$ is big by Remark 
\ref{big_rmk}. By \cite[Proposition 6.2.6 and Remark 6.2.7]{prokhorov_MSJ}, 
the morphism $\sigma$ 
is a weighted blowup with weights $a$, $b$ for some local parameters $s$, 
$t$ of $\sO_{\pr^2, p}$, where $a$, $b\in\Z_{>0}$, $a\geq b$ and $a$, $b$ are mutually 
prime. (Note that the variety $Y$ is not a toric variety in general.) 
We know that $(F^{\cdot 2})_Y=-1/(ab)$ and $A(F)=a+b$. 

Let $\pi\colon\tilde{Y}\to Y$ be the minimal resolution of $Y$ and let $E_1\subset 
\tilde{Y}$ be the strict transform of the exceptional divisor of the ordinary blowup 
of $p\in\pr^2$. Then we can check that $\ord_{E_1}\pi^*F=1/a$. 
Let $\hat{l}\subset\tilde{Y}$ be the strict transform of a general line on $\pr^2$ 
passing though $p\in\pr^2$. Since $\hat{l}$ is movable, we have 
\[
0\leq\left(\pi^*(\sigma^*\left(-K_{\pr^2}-\tau(F)F)\right)
\cdot\hat{l}\right)=3-\tau(F)\cdot
\frac{1}{a}. 
\]
This implies that $\tau(F)\leq 3a$. 
We can write $K_{\tilde{Y}}=\pi^*K_Y-E$ for some effective and $\pi$-exceptional 
$\Q$-divisor $E$ on $\tilde{Y}$ (see \cite[Corollary 4.3]{KoMo}). 
Thus $-K_{\tilde{Y}}$ is big. This implies that $\tilde{Y}$ and $Y$ are 
Mori dream spaces in the sense of \cite{HK} by \cite[Theorem 1]{TVAV}. 
In particular, $F$ is dreamy over $(\pr^2, 0)$ (see \cite[Lemma 4.8]{ELMNP06}). 

Let us set
\[
\varepsilon(F):=\max\{\varepsilon\in\R_{\geq 0}\,|\,\sigma^*(-K_{\pr^2})-\varepsilon F
\text{ nef}\}.
\]
Then $\varepsilon(F)\in(0,\tau(F)]$.

\begin{claim}\label{et_claim}
\begin{enumerate}
\renewcommand{\theenumii}{\roman{enumii}}
\renewcommand{\labelenumii}{(\theenumii)}
\item\label{et_claim1}
We have the equality $\varepsilon(F)\tau(F)=9ab$. 
\item\label{et_claim2}
We get 
\[
\vol_{\pr^2}(-K_{\pr^2}-xF)=\begin{cases}
9\left(1-\frac{x^2}{\varepsilon(F)\tau(F)}\right) & \text{if }x\in[0, \varepsilon(F)],\\
9\frac{(\tau(F)-x)^2}{\tau(F)(\tau(F)-\varepsilon(F))}
& \text{if }x\in(\varepsilon(F), \tau(F)].
\end{cases}
\]
\item\label{et_claim3}
We have the equality 
\[
\hat{\beta}(F)=1-\frac{\varepsilon(F)+\tau(F)}{3(a+b)}.
\]
\end{enumerate}
\end{claim}

\begin{proof}[Proof of Claim \ref{et_claim}]
The assertion 
\eqref{et_claim3} follows from \eqref{et_claim1} and \eqref{et_claim2}. We prove 
\eqref{et_claim1} and \eqref{et_claim2}. 

We know that 
\[
\vol_{\pr^2}(-K_{\pr^2}-xF)=\left(\left(\sigma^*(-K_{\pr^2})-xF\right)^{\cdot 2}\right)
=9-\frac{x^2}{ab}
\]
for $x\in[0, \varepsilon(F)]$. If $\varepsilon(F)=\tau(F)$, then 
$\vol_{\pr^2}(-K_{\pr^2}-\varepsilon(F)F)=0$. Thus $\varepsilon(F)^2=9ab$. 
Hence we can assume that $\varepsilon(F)<\tau(F)$. In this case, 
$\varepsilon(F)\in\Q$ and the divisor $\sigma^*(-K_{\pr^2})-\varepsilon(F)F$ gives 
a nontrivial birational contraction morphism $\mu\colon Y\to Z$ since 
$F$ is dreamy over $(\pr^2, 0)$ (see \cite[Lemma 3.1 (4)]{fjtb}). 
Moreover, since the Picard rank of $Z$ is one, $\mu_*(\sigma^*(-K_{\pr^2}))$ and 
$\mu_*F$ are numerically proportional. Thus, for $x\in[\varepsilon(F), \tau(F)]$, 
we can write 
\[
\vol_{\pr^2}(-K_{\pr^2}-xF)=\left(\mu_*\left(\sigma^*(-K_{\pr^2})
-xF\right)^{\cdot 2}\right)_Z=c(\tau(F)-x)^2
\]
for some $c\in\R_{>0}$. 
Note that 
\[
\vol_{\pr^2}(-K_{\pr^2}-\varepsilon(F)F)=9-\frac{\varepsilon(F)^2}{ab}
=c(\tau(F)-\varepsilon(F))^2
\]
and, by \cite[Theorem A]{BFJ}, 
\[
\frac{d}{dx}\bigg|_{x=\varepsilon(F)}
\vol_{\pr^2}(-K_{\pr^2}-xF)=-\frac{2\varepsilon(F)}{ab}
=-2c(\tau(F)-\varepsilon(F)).
\]
This implies that $ab=\varepsilon(F)\tau(F)/9$ 
and $c=9/(\tau(F)(\tau(F)-\varepsilon(F)))$.
\end{proof}

Since $\varepsilon(F)\leq\tau(F)\leq 3a$ and $\varepsilon(F)\tau(F)=9ab$, 
we have $\tau(F)\in[3\sqrt{ab}, 3a]$. Moreover, 
\[
\varepsilon(F)+\tau(F)=\tau(F)+\frac{9ab}{\tau(F)}
\]
and the function $x+9ab/x$ is monotonically increasing on $x\in[3\sqrt{ab}, 3a]$. 
Thus $\varepsilon(F)+\tau(F)\leq 9ab/(3a)+3a=3(a+b)$. 
This implies that 
\[
\hat{\beta}(F)=1-\frac{\varepsilon(F)+\tau(F)}{3(a+b)}\geq 0.
\]

As a consequence, we have completed the proof of Corollary \ref{P2_cor}.

\end{document}